\documentclass[10pt]{article}
\usepackage[all]{xy}
\usepackage{amsfonts,amsmath,oldgerm,amssymb,amscd}
\newcommand{\ra}{\rightarrow}		
\newcommand{\by}[1]{\stackrel{#1}{\ra}}

\newcommand{\ol}{\overline}		
\newcommand{\iso}{\by \sim}

\newtheorem{theorem}{Theorem}[section]
\newtheorem{proposition}[theorem]{Proposition}

\newtheorem{corollary}[theorem]{Corollary}

	\newcommand{\p}{\mbox{$\mathfrak p$}}

\newcommand{\ot}{\mbox{\,$\otimes$\,}}	\newcommand{\op}{\mbox{$\oplus$}}

\newcommand{\Aut}{\mbox{\rm Aut\,}}

\newcommand{\Um}{\mbox{\rm Um}}		\newcommand{\SL}{\mbox{\rm SL}}
\newcommand{\GL}{\mbox{\rm GL}}

\begin{document}

\begin{center}
{\large \bf Projective modules over discrete Hodge algebras}\\\vspace{.2in} 
        {\large Manoj  Kumar Keshari}\\
\vspace{.1in}
{\small
Department of Mathematics, IIT Mumbai, Mumbai - 400076, India;\;
        keshari@math.iitb.ac.in}
\end{center}

\section{Introduction}
All the rings are assumed to be commutative Noetherian and all
the modules are finitely generated.

Let $A$ be a ring. In (\cite{Vorst}, Theorem 1.1), Vorst proved that
if all projective modules over polynomial extensions of $A$ are
extended from $A$, then all projective modules over discrete Hodge
$A$-algebras are extended from $A$ (An $A$-algebra $R$ is a {\it
discrete Hodge $A$-algebra} if $R=A[X_0,\ldots,X_n]/I$, where $I$ is
an ideal generated by monomials). In this note, we extend the above
result of Vorst by proving the following result.

\begin{theorem}\label{k1}
Let $A$ be a ring and $r>0$ be an integer. Assume that all projective
modules of rank $r$ over polynomial extensions of $A$ are extended
from $A$. Then all projective modules of rank $r$ over discrete Hodge
$A$-algebras are extended from $A$.
\end{theorem} 

We note that Lindel gave another proof of Vorst's result (\cite{lin},
Theorem 1.5) and a proof of (\ref{k1}) is implicit in Lindel's
proof. But the idea of our proof is different from Lindel's and it also
gives other results which we describe below.

Let $A$ be a ring of dimension $d$ and let $r>d/2$. Assume that $A$ is
of finite characteristic prime to $r!$. In
(\cite{Roitman}, Theorem 5), Roitman proved that if $P$ is a
projective module of rank $r$ over $R=A[X_1,\ldots,X_n]$ such that
$P\op R$ is extended from $A$, then $P$ is extended from $A$. In
particular, if $A$ is a local ring of dimension $d$, characteristic of
$A$ is positive and prime to $d!$, then all stably free modules of
rank $> d/2$ over polynomial extensions of $A$ are free. 

We will prove the following analogue of Roitman's result for discrete
Hodge $A$-algebras.

\begin{theorem}\label{k2}
Let $A$ be a ring of dimension $d$. Assume $A$ is of finite
characteristic prime to $r!$. Let $R$ be a discrete Hodge $A$-algebra
and let $P$ be a projective $R$-module of rank $r> d/2$. If $P\op R$
is extended from $A$, then $P$ is extended from $A$.
\end{theorem}

As a corollary to the above result, if $A$ is a local ring of
dimension $d$, characteristic of $A$ is finite and prime to $d!$,
then all stably free modules of rank $> d/2$ over discrete Hodge
$A$-algebras are free. 
\medskip

Now, we will describe our last result.  Let $A$ be a ring of dimension
$d$ and let $R=A[X_1,\ldots,X_n]$.  In (\cite{Wiemers}, Section 4),
Wiemers asked the following question: Is the natural map $\Um_r(R) \ra
\Um_r(R/(X_1X_2\ldots X_k ))$ surjective for all $r$ and $1\leq k\leq
n$?

Wiemers (\cite{Wiemers}, Proposition 4.1) answered the above question
in affirmative when $r\geq d+2$ or $r=d+1$ and $1/d!\in A$.  We will
prove the following result which gives a partial answer to Wiemers
question in affirmative.

\begin{theorem}\label{k3}
Let $A$ be a ring of dimension $d$. Assume characteristic of $A$ is
positive and prime to $d!$. Let $R=A[X_1,\ldots,X_n]$ and let
$I\subset J$ be two ideals of $R$ generated by square free
monomials. Then the map $\Um_r(R/I) \ra \Um_r(R/J)$ is
surjective for $r \geq \frac{d}{2}+2$.
\end{theorem}

\section{Preliminaries}

Given a cartesian diagram of rings $$\xymatrix{ A\ar [r] \ar [d] & A_1
\ar [d]_{j_1} \\ A_2 \ar [r]_{j_2} & A_0 }$$ where $j_2$ is a surjective
map. If $P$ is a projective $A$-module, then the above diagram induces
a cartesian diagram (\cite{Milnor}, Section 2)
$$\xymatrix{ P\ar [r] \ar [d] & P_1
\ar [d] \\ P_2 \ar [r] & P_0 } $$ 
where $P_i=P\ot A_i$ for $i=0,1,2$.
\medskip

We begin by stating the following two results of A. Wiemers
(\cite{Wiemers}, Proposition 2.1 and Theorem 2.3) respectively.

\begin{proposition}\label{w1}
Given a cartesian square of rings with $j_2$ surjective and a
projective $A$-module $P$. Then

$(i)$ If $\Aut_{A_2}(P_2) \ra \Aut_{A_0}(P_0)$ is surjective, then so is 
$\Aut_A(P) \ra \Aut_{A_1}(P_1)$.

$(ii)$ If $\Aut_{A_2}(P_2) \ra \Aut_{A_0}(P_0)$ is surjective and
$Q\ot_A A_i \iso P_i$, $i=1,2$ for another projective $A$-module
$Q$, then $P\iso Q$. In particular, if $P_1$ and $P_2$ have the
cancellation property, then so does $P$.

$(iii)$ Let, in addition, $j_1$ be surjective. If $\Um(P_2) \ra
\Um(P_0)$ is surjective, then so is $\Um(P) \ra \Um(P_1)$.
\end{proposition}

\begin{theorem}\label{w3}
Let $A$ be a ring and let $J$ be an ideal of $R=A[X_1,\ldots,X_n]$
generated by square free monomials. Then the natural map $\GL_r(R) \ra
\GL_r(R/J)$ is surjective.
\end{theorem}

Given a simplicial subcomplex $\Sigma$ of $\Delta_n$ and a ring $A$,
let $I(\Sigma)$ be the ideal of $A[X_0,\ldots,X_n]$ generated by all
square free monomials $X_{i_1}X_{i_2}\ldots X_{i_k}$ with $0\leq
i_1<i_2<\ldots <i_k\leq n$ and $\{i_1,\ldots,i_k\}$ is not a face of
$\Sigma$. By $A(\Sigma)$, we denote the discrete Hodge $A$-algebra
$A[X_0,\ldots,X_n]/I(\Sigma)$. 
\medskip

The following result is due to Vorst
(\cite{Vorst}, Lemma 3.4) and is very crucial for the proof of our
results.

\begin{proposition}\label{v1}
Let $\Sigma$ be a simplicial subcomplex of $\Delta_n$ which is not a
simplex. Then there exists an $i\in \{ 0,1,\ldots,n\}$ and simplicial
subcomplexes $\Sigma_2 \subset \Sigma_1 \subset \Sigma$ such that we
have a cartesian square of rings
$$ \xymatrix{ A(\Sigma) \ar [r]^{i_1} \ar [d]_{i_2} & A(\Sigma_1)\ar
[d]^{j_1} \\ A(C(\Sigma_2)) \ar[r]_{j_2} & A(\Sigma_2) }$$ where all
maps are natural surjections and $\Sigma_2 \subset \Sigma_1 \subset
{\Sigma^i}$, where ${\Sigma^i}$ is the $n-1$ simplex of which $i$ is
not a vertex and $C(\Sigma_2)$ is the cone on $\Sigma_2$ with vertex
$i$. Note that $j_2$ is a split surjection and
$A(C(\Sigma_2))=A(\Sigma_2)[X_i]$.
\end{proposition}

We end this section by stating two results of Wiemers (\cite{Wiemers},
Theorem 3.6) and (\cite{Wiemers1}, Theorem 4.3) respectively which
will be used in section $4$.

\begin{theorem}\label{w4}
Let $A$ be a ring of dimension $d$.  Let $I\subset J$ be ideals in
$R=A[X_1,\ldots,X_n]$ generated by square free monomials. Let $P $ be
a projective module over $R/I$. If either rank $P\geq d+1$ or rank
$P\geq d$ and $1/d!\in A$, then the natural map $\Aut_{R/I}(P)\ra
\Aut_{R/J}(P/\ol JP)$ with $\ol J=J/I$ is surjective.
\end{theorem}

\begin{theorem}\label{w41}
Let $A$ be a ring of dimension $d$ with $1/d!\in A$ and
$B=A[X_1,\ldots,X_n]$. Let $P$ and $P_1$ be projective $B$-modules
of rank $\geq d$. Assume $P\op B\iso P_1 \op B$. If
$P/(X_1,\ldots,X_n)P \iso P_1/(X_1,\ldots,X_n)P_1$, then $P \iso
P_1$.  

In other words, if the projective $A$-module $P/(X_1,\ldots,X_n)P$ is
cancellative, then $P$ is cancellative.
\end{theorem}


\section{Main Theorem}

In this section we prove our main results mentioned in the introduction. 
\medskip

{\bf Proof of Theorem \ref{k1} :}
Let $B=A[X_0,\ldots,X_n]/I$ be a discrete Hodge $A$-algebra and let $P$ be
a projective $B$-module of rank $r$ (here $I$ is a monomial ideal). It
is enough to assume that $I$ is a square free monomial ideal. Then
$I=I(\Sigma)$ for some simplicial subcomplex $\Sigma$ of $\Delta_n$
and $B=A(\Sigma)$. We will use induction on $n$.

If $n=0$, then there is nothing to prove, as $A(\Sigma)=A$ or
$A[X_0]$. Let $n>0$ and assume the result for $n-1$. We will apply
(\ref{v1}). By induction hypothesis, all projective modules of rank
$r$ over $A_1=A(\Sigma_1)$ and $A_0=A(\Sigma_2)$ are extended from
$A$. Also all projective modules of rank $r$ over
$A_2=A(C(\Sigma_2))=A(\Sigma_2)[X_i]$ are extended from $A[X_i]$ and
hence are extended from $A$.

Write $P_i=P\ot_A A_i$, $i=0,1,2$. Clearly, the natural map
$\Aut_{A_2}(P_2) \ra \Aut_{A_0}(P_0)$ is surjective. Hence, if
$Q=P/(X_0,\ldots,X_n)P$, then $P_1 \iso Q\ot A_1$ and $P_2\iso Q\ot
A_2$, by induction hypothesis. Hence, by $(\ref{w1}(ii))$, $P\iso Q\ot A$,
i.e. $P$ is extended from $A$. This proves the result.  $\hfill \square$
\medskip

{\bf Proof of Theorem \ref{k2} :}
Let $R=A[X_0,\ldots,X_n]/I$ be a discrete Hodge $A$-algebra and let $P$ be
a projective $R$-module of rank $r$ (here $I$ is a monomial ideal). Again, it
is enough to assume that $I$ is a square free monomial ideal. Then
$I=I(\Sigma)$ for some simplicial subcomplex $\Sigma$ of $\Delta_n$
and $R=A(\Sigma)$. We will use induction on $n$.

When $n=0$, there is nothing to prove as $R=A$ or $A[X_0]$. Let
$n>0$ and assume the result for $n-1$. We will apply (\ref{v1}). Let
$A_1=A(\Sigma_1)$, $A_2=A(C(\Sigma_2))$ and $A_0=A(\Sigma_2)$. Write
$P_i=P\ot_A A_i$ for $i=0,1,2$.

Since $R\ra A_i$ are natural surjections, $P_i \op A_i$ are extended
from $A$, $i=1,2$. Hence, by induction hypothesis $P_i$ is extended
from $A$, $i=1,2$. Therefore, if $Q=P/(X_0,\ldots,X_n)P$,  $P_i\iso Q\ot
A_i$, $i=1,2$.  Clearly, the natural map $\Aut_{A_2}(P_2) \ra
\Aut_{A_0}(P_0)$ is surjective. Hence,  by $(\ref{w1}(ii))$, $P\iso Q\ot_A R$,
i.e. $P$ is extended from $A$. This proves the result. $\hfill \square$
\medskip

{\bf Proof of Theorem \ref{k3} :} It is enough to show that
the natural map $\Um_r(R)\ra \Um_r(R/J)$ is surjective for every ideal
$J$ of $R$ generated by square free monomials.

Let $v\in \Um_r(R/J)$. We have an exact sequence $0\ra P \ra (R/J)^r
\by v R/J \ra 0$.

Since $P\op A/J$ is free, by (\ref{k2}), $P$ is extended from $A$,
i.e. $P=\ol P \ot_A R$, where $\ol P = P/(X_1,\ldots,X_n)P$. Hence, we
have the following commutative diagram

$$
\xymatrix{ 
	0 \ar [r] & \ol P\ot_A R \ar [r] \ar [d]^\iso & (R/J)^r \quad \ar
	[r]^{v(0)\otimes R} \ar@{-->} [d]^{\sigma } & R/J  
	\ar [r] \ar [d]^{id} & 0      \\ 
	0 \ar [r] & P \ar [r] & (R/J)^r \quad \ar [r]^{v} &
	R/J \ar [r] & 0
	}
$$ where $v(0)$ is the image of $v$ in $\Um_r(A)$ under the map $R/J
\ra A$ given by $\ol X_i \mapsto 0$, $i=1,\ldots,n$.  Hence, there
exists $\sigma \in \GL_r(R/J)$ such that $v\sigma = v(0)\ot R$. By
(\ref{w3}), $\sigma$ can be lifted to $\Delta \in
\GL_r(R)$ and $v(0)\Delta^{-1}\in \Um_r(R)$ is a lift of $v$. This
proves the result.  $\hfill \square$

\section{Some Auxiliary Results}

As an application of (\ref{v1}), we will give an alternative proof of the 
following result of Wiemers (\cite{Wiemers}, Corollary 4.4).

\begin{theorem}
Let $A$ be a ring of dimension $d$ with $1/d!\in A$. Let
$B=A[X_0,\ldots,X_n]/I$ be a discrete Hodge $A$-algebra. Let $P$ be
a projective $B$-module of rank $\geq d$. If the projective $A$-module
$P/(X_0,\ldots,X_n)P$ is cancellative, then $P$ is cancellative.
\end{theorem}

\begin{proof}
If $B$ is a polynomial ring over $A$, then the result follows from
(\ref{w41}).  It is enough to assume that $I$ is generated by square
free monomials.  Hence $I=I(\Sigma)$ for some simplicial subcomplex
$\Sigma$ of $\Delta_n$. We will apply induction on $n$.

By (\ref{v1}), we have the following cartesian square 

$$ \xymatrix{ A(\Sigma) \ar [r]^{i_1} \ar [d]_{i_2} & A(\Sigma_1)\ar
[d]^{j_1} \\ A(C(\Sigma_2)) \ar[r]_{j_2} & A(\Sigma_2) .}$$

By (\ref{w4}), the natural map $\Aut_{A(C(\Sigma_2))}(P\ot
A(C(\Sigma_2))) \ra \Aut_{A(\Sigma_2)}(P\ot A(\Sigma_2))$ is
surjective and by induction hypothesis on $n$, $P\ot A(C(\Sigma_2))$ and
$P\ot A(\Sigma_1)$ are cancellative. Hence, by $(\ref{w1}(ii))$, $P$ is
cancellative. This proves the result. $\hfill \square$
\end{proof}

\begin{theorem}\label{w42}
Let $A$ be a ring of dimension $d$ with $1/d! \in A$ and
$R=A[X_1,\ldots,X_n]$. Let $P$ be a projective $R$-module of rank $d$
such that $P\op R$ is extended from $A$. Then $P$ is
extended from $A$.
\end{theorem}

\begin{proof}
By Quillen's local-global principle (\cite{Quillen}, Theorem 1), it is
enough to assume that $A$ is local.  Then $P\op R$ is free.  Since
$P/(X_1,\ldots,X_n)P$ is free, by (\ref{w41}), $P$ is free. This
proves the result.  $\hfill \square$
\end{proof}
\medskip

\begin{remark}
When $P$ is stably free, the above result (\ref{w42})
is due to Ravi A Rao (\cite{Ravi} Corollary 2.5). More precisely, Rao
proved that if $A$ is a ring of dimension $d$ with $1/d! \in A$, then
every $v\in \Um_{d+1}(A[X])$ is extended from $A$, i.e. there exists
$\sigma \in \SL_{d+1}(A[X])$ such that $v \sigma= v(0)$.
\end{remark}
\medskip

Following the proof of (\ref{k2}) and using (\ref{w42}), we get the
following:

\begin{corollary}\label{manoj}
Let $A$ be a ring of dimension $d$ with $1/d!\in A$. Let $B$ be a
discrete Hodge $A$-algebra. Let $P$ be a projective $B$-module of
rank $d$ such that $P\op B$ is extended from $A$, then $P$ is extended
from $A$. In particular, every stably free $B$-module of rank $d$ is
extended from $A$.
\end{corollary}

During CAAG VII meeting in Bangalore, Kapil H Paranjape asked if we
can extend the above results (\ref{k1}, \ref{k2}, \ref{manoj}) for
locally discrete Hodge $A$-algebras (Definition: A positively graded
$A$-algebra $B$ is a locally discrete Hodge $A$-algebra if $B_{\p}$ is a
discrete Hodge $A_{\p}$-algebra for every prime ideal $\p$ of $A$. The
answer is yes and follows from the following result of Lindel
(\cite{lin}, Theorem 1.3) which generalises Quillen's patching theorem
(\cite{Quillen}, Theorem 1) from polynomial rings to positively graded
rings.

\begin{theorem}
Let $A$ be a ring and let $M$ be a finitely presented module over a
positively graded ring $R=\op_{i\geq 0}\, R_i$, $R_0=A$. Then the set
$J(A,M)$, of all $u\in A$ for which $M_u$ is extended from $A_u$, is
an ideal of $A$.

In particular, if $M_{\p}$ is extended from $A_{\p}$ for all prime ideal
$\p$ of $A$, then $M$ is extended from $A$.
\end{theorem}


{}

\end{document}